\documentclass[11pt,oneside,leqno,english]{amsart}
\pdfoutput=1
\usepackage[utf8]{inputenc}
\usepackage[english]{babel}
\usepackage[T1]{fontenc}
\usepackage{txfonts,microtype}
\usepackage{graphicx}
\usepackage[all]{xy}
\usepackage{framed,enumitem,stackrel}
\usepackage[svgnames,table]{xcolor}
\usepackage{amssymb,amsbsy,amsfonts,amssymb,amscd,mathrsfs,amsmath,xspace,bm}
\newcommand{\C}{\mathbb{C}\xspace}  \renewcommand{\P}{\mathbb{P}\xspace} \newcommand{\N}{\mathbb{N}\xspace} \newcommand{\Z}{\mathbb{Z}\xspace} \newcommand{\Q}{\mathbb{Q}\xspace}%bold letters
\newcommand{\widebar}[1]{\mkern 1.5mu\overline{\mkern-1.5mu#1\mkern-1.5mu}\mkern 1.5mu}%wide bar
\renewcommand{\leq}{\leqslant} \renewcommand{\geq}{\geqslant}%comparison
\newcommand{\bydef}{\mathrel{\mathop:}=} \newcommand{\defby}{=\mathrel{\mathop:}}%definition
\DeclareMathOperator{\moinsun}{-1}%small -1
\DeclareMathOperator{\rk}{rk}%rank
\newcommand{\abs}[1]{\lvert#1\rvert} %absolute values
\newcommand{\smbullet}{{\scriptscriptstyle \bullet}}%small bullet
\newcommand{\llangle}{\langle\!\langle\xspace} \newcommand{\rrangle}{\rangle\!\rangle\xspace}%iterated Laurent series
\newcommand{\xym}[3][1]{\left.\vcenter{\xy\xymatrix"m"@R=.5pt@C=.5pt@W=1em@H=#1em{#3}\POS"m1,1"."m#2"!C*\frm{(}*\frm{)}\endxy}\right.}% pretty matrix
\numberwithin{equation}{section}

\usepackage{ifmtarg}
\makeatletter
\newtheoremstyle{myplain}
{12pt}{12pt}{\itshape}{}{\bfseries}{.}{ }
{\thmnumber{(\textup{#2})} \thmname{#1}\@ifmtarg{#3}{}{\,(\mdseries\thmnote{#3})}}
\makeatother
\theoremstyle{myplain}
\newtheorem{Prop}[equation]{Proposition}
\newtheorem{Thm}[equation]{Theorem}
\newtheorem*{THM}{Main Result}
\newtheorem{Lem}[equation]{Lemma}
\title{Fiber Integration on the Demailly Tower}
\author{Lionel~Darondeau}
\address{Laboratoire de Mathématiques d'Orsay\\Université Paris-Sud (France).}
\curraddr{Institut de Recherche Mathématique Avancée\\Université de Strasbourg (France).}
\thanks{To appear in \textbf{Annales de l'Institut Fourier}.}
\email{lionel.darondeau@normalesup.org}
\date{January 28, 2015}
\keywords{Demailly tower of logarithmic directed manifold. Segre classes. Iterated Laurent series.}
\begin{document}
\begin{abstract}
  The goal of this work is to provide a fiber integration formula on the Demailly tower, that avoids step-by-step elimination of horizontal cohomology classes, and that yields computational effectivity. 
  A natural twist of the Demailly tower is introduced and a recursive formula for the total Segre class at \(k\)-th level is obtained. 
  Then, by interpreting single Segre classes as coefficients, an iterated residue formula is derived. 
\end{abstract}
\maketitle

\section{Introduction}
Let \(\widebar{X}\) be a complex manifold and let \(D\) be a divisor on \(\widebar{X}\) with \textsl{normal crossings}, that is
\(D=\sum D_i\),
where the components \(D_i\) are smooth irreducible divisors that meet transversally.
For such a pair \(\bigl(\widebar{X},D\bigr)\), one denotes by \(T_{\widebar{X}}\bigl(-\log D\bigr)\) the \textsl{logarithmic tangent bundle  of \(\widebar{X}\) along \(D\)} (\cite{MR859200}).

Given a subbundle: 
\[
  V
  \subset 
  T_{\widebar{X}}(-\log D)
  \subset 
  T_{\widebar{X}}
\] 
of the logarithmic tangent bundle, one constructs (\cite{MR1492539,MR1824906}), for any fixed order \(\kappa\in\N\), the \textsl{logarithmic Demailly tower} of projectivized bundles:
\begin{multline*}
  \bigl(\widebar{X}_\kappa,D_\kappa,V_{\kappa}\bigr)
  \to 
  \bigl(\widebar{X}_{\kappa-1},D_{\kappa-1},V_{\kappa-1}\bigr)
  \to\dotso\\
  \dotso
  \to 
  \bigl(\widebar{X}_1,D_1,V_1)
  \to 
  \bigl(\widebar{X}_0,D_0,V_0)
  \bydef
  \bigl(\widebar{X},D,V\bigr),
\end{multline*}
having the main property that every holomorphic map \(g\colon\C\to \widebar{X}\setminus D\) lifts as maps:
\[
  g_{[i]}
  \colon
  \C\to \widebar{X}_i\setminus D_i
  \qquad{\scriptstyle(i=0,1,\dotsc,\kappa)},
\]
which depends only on the corresponding \(i\)-jet of \(g\).
Later on in section §2, we will describe precisely this construction, central here.

For any two integers \(j,k\in0,1,\dotsc,\kappa\), the composition of the projections \(\pi_i\colon \widebar{X}_i\to \widebar{X}_{i-1}\) yields a natural projection from the \(j\)-th level of the Demailly tower to the lower \(k\)-th level:
\[
  \pi_{j,k}
  \bydef
  \pi_{k+1}\circ\dotsb\circ\pi_{j}
  \colon
  \widebar{X}_{j}\to \widebar{X}_k.
\]

\medskip
The Demailly tower is of great importance in the study of the algebraic degeneracy of entire curves on \(\widebar{X}\setminus D\) (\textit{cf.} the enlightening surveys \cite{div-rou-survey,paun-survey}). 
A first step towards the proof of algebraic degeneracy of entire curves is to prove the existence of a non zero polynomial \(P\) on \(\widebar{X}\) such that every non constant entire curve \(g\colon\C\to\widebar{X}\setminus D\) satisfies the algebraic differential equation:
\[
  P_{g(t)}\bigl(g'(t),g''(t),\dotsc,g^{(\kappa)}(t)\bigr)
  =
  0,
  \quad
  \text{for all}
  \ t\in\C.
\]

Being by definition a projective vector bundle, the manifold \(\widebar{X}_i\) comes naturally equipped with a tautological line bundle, \(\mathcal{O}_{\widebar{X}_i}(-1)\),
the multiples of which are usually denoted by
\(\mathcal{O}_{\widebar{X}_{i}}(m)\bydef(O_{\widebar{X}_{i}}(-1)^{\vee})^{\otimes m}\).
The direct image:
\[
\mathcal{O}\left(E_{\kappa,m}(V_0)^{\star}(\log D_0)\right)
  \bydef
  (\pi_{\kappa,0})_{\star}\mathcal{O}_{\widebar{X}_{\kappa}}(m)
\] 
is the sheaf of sections of a holomorphic bundle \(E_{\kappa,m}(V_{0})^{\star}(\log D_{0})\),
called the \textsl{Demailly-Semple bundle of jet differentials},
and a fundamental vanishing theorem (\cite{MR1492539,MR1824906}) states that for every global section:
\[
  P
  \in
  H^{0}\bigl(
  \widebar{X}_{\kappa},
  \mathcal{O}_{\widebar{X}_{\kappa}}(m)
  \otimes
  \pi_{\kappa,0}^{\star}
  A^{\vee}
  \bigr)
  \simeq
  H^{0}\bigl(
  \widebar{X},
E_{\kappa,m}(V_{0})^{\star}(\log D_{0})
  \otimes
  A^{\vee}
  \bigr),
\]
with values in the dual \(A^{\vee}\) of an ample line bundle \(A\to\widebar{X}\),  
one has as desired, for any \(\kappa\)-jet \((g',g'',\dotsc,g^{(\kappa)})\) of non constant entire map \(g\colon\C\to\widebar{X}\setminus D\):
\[
  P_{g(t)}\bigl(g'(t),g''(t),\dotsc,g^{(\kappa)}(t)\bigr)
  =
  0,
  \quad
  \text{for all}
  \ t\in\C.
\]
One has thus to ensure the existence of global sections of the line bundle \(\mathcal{O}_{\widebar{X}_{\kappa}}(m)\otimes\pi_{\kappa,0}^{\star}A^{\vee}\), possibly with \(m\gg1\). 

One approach, with Schur bundles (\cite{MR1492539}), consists in bounding positive even cohomology groups \(H^{2i}\) in order to use the Riemann-Roch theorem.
In \cite{MR2257847}, in dimension \(3\), the author is able to bound the dimension of \(H^2\) by use of the famous algebraic Morse inequalities (\cite{MR1492539,MR1339712}). Later in \cite{arXiv1005.0405}, the case of arbitrary dimension is completed, for high order jet differentials.

With a different approach, in \cite{MR2918158} the case of arbitrary dimension is completed by use of a stronger version of algebraic Morse inequalities.

Another approach, developed in various contexts (\cite{arxiv:1011.4710,COM:9073852,arXiv:1402.1396,MR2441250,MR2495771,MR2593279}), consists in applying the holomorphic Morse inequalities in order to prove the existence of sections of a certain more tractable subbundle of the bundle of jet differentials.
One is led to establish the positivity of a certain intersection number on the \(\kappa\)-th level of the Demailly tower:
\[
  I
  =
  \int_{\widebar{X}_\kappa}
  f\Bigl(
  c_1\bigl(\pi_{\kappa,1}^{\star}\mathcal{O}_{\widebar{X}_1}(1)\bigr),
  \dotsc,
  c_1\bigl(\pi_{\kappa,\kappa-1}^{\star}\mathcal{O}_{\widebar{X}_{\kappa-1}}(1)\bigr),
  c_1\bigl(\mathcal{O}_{\widebar{X}_\kappa}(1)
  \bigr)\Bigr),
\]
where \(f\) is a polynomial of large degree:
\[
  \deg(f)
  =
  n+\kappa\,(\rk\,P(V))
  =
  \dim(\widebar{X}_{\kappa}),
\] 
in the first Chern classes \(c_1\bigl(\pi_{\kappa,i}^\star\mathcal{O}_{\widebar{X}_{i}}(1)\bigr)\).
\medskip

When computing this intersection number, the standard strategy is to integrate along the fibers of the projections \(\pi_{i,i-1}\colon \widebar{X}_i\to \widebar{X}_{i-1}\), until one obtains an intersection product on the basis \(\widebar{X}_0\), where the intersection of cohomology classes becomes simpler.

In \cite{MR2593279}, the authors use step-by-step elimination of Chern classes, and are able to disentangle the complex intrication between horizontal and vertical cohomology classes by a technical tour de force. These precise computations yield \emph{effectivity}. 

In \cite{COM:9073852}, the author makes a clever use of \emph{Segre classes} in order to avoid a large part of the computations, but on the other hand, effectivity cannot be reached.

In \cite{arxiv:1011.4710}, the author uses equivariant geometry in order to prove a \emph{residue formula} in several variables, that avoids step-by-step elimination and yields effectivity.

In the present paper, we combine ideas coming from these authors, in order to prove a similar residue formula in several variables, that is valid in a versatile geometric context, since it holds in any situation where the Demailly tower appears, \textit{cf. e.g.} \cite{MR2911888,COM:9073852}.
Our proof borrows from \cite{COM:9073852} the technical simplification of the use of Segre classes, it yields computational effectivity as in \cite{MR2593279}, and it is in the very spirit of the formula of \cite{arxiv:1011.4710}.
\medskip

To enter into the details, 
by the Leray-Hirsch theorem (\cite{hatcher}), the cohomology ring
\(H^{\smbullet}\bigl(\widebar{X}_{\kappa}\bigr)\) of \(\widebar{X}_\kappa\) is the free module generated 
by the first Chern classes \(c_1\bigl(\pi_{\kappa,i}^\star\mathcal{O}_{\widebar{X}_i}(-1)\bigr)\) over
the cohomology ring \(H^\smbullet\bigl(\widebar{X}_0\bigr)\) of the basis \(\widebar{X}_0\),
but the implementation of the computation in \cite{arXiv:1402.1396}
suggests to naturally consider a different basis for the vertical cohomology by introducing the line bundles:
\[
  L_i
  \bydef
  \mathcal{O}_{\widebar{X}_i}(-1)\otimes
  \pi_{i,i-1}^{\star}\mathcal{O}_{\widebar{X}_{i-1}}(-1)\otimes 
  \dotsb\otimes
  \pi_{i,1}^{\star}\mathcal{O}_{\widebar{X}_0}(-1)
  \qquad
  {\scriptstyle(i=1,\dotsc,\kappa)}.
\]
We will use the notation \(v_i\) for the first Chern class of the dual of this line bundle \(L_i\) (dropping the pullbacks):
\[
  v_i
  \bydef
  c_1\bigl(L_i^{\vee}\bigr)
  =
  c_1\bigl(\mathcal{O}_{\widebar{X}_i}(1)\bigr)+
  \dotsb+
  c_1\bigl(\mathcal{O}_{\widebar{X}_1}(1)\bigr)
  \qquad
  {\scriptstyle(i=1,\dotsc,\kappa)}.
\]

Note that this formula looks like a plain change of variables having inverse:
\[
  c_1\bigl(\mathcal{O}_{\widebar{X}_i}(1)\bigr)
  =
  v_i-v_{i-1}
  \qquad{\scriptstyle(i=2,\dotsc,\kappa)},
\]
thus, clearly, the polynomial \(f\) appearing in the intersection product \(I\) above
has also a polynomial expression in terms of \(v_1,\dotsc,v_\kappa\).
We will shortly provide a formula in order to integrate a polynomial under this new form, still denoted \(f\).

\medskip
\label{par:multivariate formal series}
Let \(K\) be a field. A \textsl{multivariate formal series} in \(\kappa\) variables with coefficients in \(K\) is a collection
of coefficients in \(K\), indexed by \(\Z^{\kappa}\): 
\[
  \varPsi
  \colon
  \Z^{\kappa}\to K.
\] 
The space of formal series is naturally a \(K\)-vector space.

In analogy with polynomials, it is usual to denote, without convergence consideration:
\[
  \varPsi(t_1,\dotsc,t_\kappa)
  \bydef
  \sum_{i_1,\dotsc,i_\kappa\in\Z}
  \varPsi(i_1,\dotsc,i_\kappa)\,
  {t_1}^{i_{1}}\dotsm{t_\kappa}^{i_{\kappa}},
\]
hence, in order to avoid confusion, we will write:
\[
  \bigl[t_1^{i_1}\dotsm t_{\kappa}^{i_{\kappa}}\bigr]
  \Bigl(
  \varPsi(t_1,\dotsc,t_\kappa)
  \Bigr)
  \bydef
  \varPsi(i_1,\dotsc,i_\kappa),
\]
to extract the coefficient indexed by \(i_1,\dotsc,i_\kappa\), that is the coefficient of the monomial \(t_1^{i_1}\dotsm t_\kappa^{i_\kappa}\) in the expansion of \(\varPsi(t_1,\dotsc,t_\kappa)\).
The \textsl{support} of the formal series \(\varPsi\)
is the subset of indices at which \(\varPsi\) is non zero:
\[
  \mathrm{supp}(\varPsi)
  \bydef
  \bigl\{
    \underline{i}\in\Z^{\kappa}
    \colon
  [t_{1}^{i_{1}}\dotsm t_{\kappa}^{i_{\kappa}}]  
  \big(\varPsi\big)
    \neq0
  \bigr\}.
\] 

\medskip

One defines the \textsl{Cauchy product} \(\varPsi_{1}\varPsi_{2}\) of two formal power series:
\begin{equation}
  \label{eq:CauchyProduct}
  \varPsi_{1}\varPsi_{2}
  \colon
  (i_1,\dotsc,i_\kappa)
  \mapsto
  \sum_{\underline{j}+\underline{k}=\underline{i}}
  [t_{1}^{j_{1}}\dotsm t_{\kappa}^{j_{\kappa}}]\big(\varPsi_{1}\big)\,
  [t_{1}^{k_{1}}\dotsm t_{\kappa}^{k_{\kappa}}]\big(\varPsi_{2}\big),
\end{equation}
whenever the displayed sum is a finite sum for each \(\kappa\)-tuple:
\[
  \underline{i}
  \bydef
  i_1,\dotsc,i_{\kappa}.
\]
For a fixed partial ordering on \(\Z^\kappa\), when considering only the series having well ordered support, the Cauchy product of two such series is always meaningful,
since the computation of the coefficient of each monomial involves only finitely many terms.
Moreover, for each choice of partial ordering, the set of formal series having well ordered support, equipped with the Cauchy product, forms a field (\cite[Theorem 13.2.11]{MR470211}).

We give two examples of such fields.
A \textsl{multivariate Laurent series} is a multivariate formal series, the support of which is well ordered for the standard product order on \(\Z^{\kappa}\).
An \textsl{iterated Laurent series} is a multivariate formal series, the support of which is well ordered for the lexicographic order on \(\Z^{\kappa}\).
The field of iterated Laurent series is an extension of the field of multivariate Laurent series.

After several Laurent expansions at the origin, any rational function becomes an iterated Laurent series (but not necessarily a multivariate Laurent series), as we will explain in more details later in §3.

We come back to the subbundle \(V=V_0\subset T_{\widebar{X}_0}\bigl(-\log D_0\bigr)\). The \textsl{total Segre class} of this bundle \(V_0\to\widebar{X}_0\):
\[
  s_{\smbullet}(V_0)
  =
  1
  +
  s_1(V_0)
  +
  s_2(V_0)
  +
  \dotsb
  +
  s_{\dim(\widebar{X}_0)}(V_0)
\]
is the inverse of the total Chern class of \(V_0\) in \(H^{\smbullet}(\widebar{X}_0)\). This notion is strongly related to integration along the fibers of a projective vector bundle (\cite{MR1644323}). We will be more explicit about this relation below in §3. 

We are now in position to state the main result of this work. 
Recall for \(i=1,\dotsc,\kappa\) the notation \(v_{i}\bydef c_{1}(L_{i}^{\vee})\),
set:
\[
  r
  \bydef
  \rk\,P(V)
  =
  \rk\,V-1,
\]
and introduce the (finite) generating series:
\[
  s_u(V_0)
  =
  1+
  u\,s_1(V_0)
  +
  u^2\,s_2(V_0)
  +
  \dotsb
  +
  u^{\dim(\widebar{X}_0)}\,
  s_{\dim(\widebar{X}_0)}(V_0).
\]
\begin{THM}
  For any polynomial:
  \[
    f\in H^{\bullet}(\widebar{X}_0,V_{0})[t_1,\dotsc,t_{\kappa}],
  \]
  in \(\kappa\) variables \(t_1,\dotsc,t_{\kappa}\), with coefficients in the cohomology ring \(H^{\bullet}(\widebar{X}_0,V_{0})\), 
  the intersection number:
  \[
    I
    \bydef
    \int_{\widebar{X}_{\kappa}}
    f(v_{1},\dotsc,v_{\kappa})
  \]
  is equal to the Cauchy product coefficient:
  \[
    I
    =
    \bigl[
      t_1^{r}\dotsm t_{\kappa}^{r}
    \bigr]
    \biggl(
    \Phi_{\kappa}\bigl(t_{1},\dotsc,t_{\kappa}\bigr)\,
    I(t_{1},\dotsc,t_{\kappa})
    \biggr),
  \]
  where \(\Phi_\kappa(t_{1},\dotsc,t_{\kappa})\) is the universal rational function:
  \[
    \Phi_{\kappa}
    (t_1,\dotsc,t_\kappa)
    =
    \prod_{1\leq i<j\leq\kappa}
    \frac{t_j-t_i}{t_j-2\,t_i}
    \;
    \prod_{2\leq i<j\leq\kappa}
    \frac{t_j-2\,t_i}{t_j-2\,t_i+t_{i-1}},
  \]
  and where \(I(t_{1},\dotsc,t_{\kappa})\) is the multivariate Laurent polynomial involving only explicit data of the base manifold:
  \[
    I(t_1,\dotsc,t_\kappa)
    =
    \int_{\widebar{X}_0}
    f(t_1,\dotsc,t_\kappa)\,
    s_{1/t_1}(V_0)\dotsm s_{1/t_{\kappa}}(V_0).
  \]
\end{THM}

Concretely, the computation of this intersection number \(I\), on the \(\kappa\)-th level \(\widebar{X}_{\kappa}\), can be brought down to the basis \(\widebar{X}_0\) as follows:
\begin{description}[wide]
  \item[$\square$\,\scshape Step 1a]
    Compute on the basis \(\widebar{X}_0\) the intersection number with parameters \(t_1,\dotsc,t_\kappa\):
\[
  I(t_1,\dotsc,t_\kappa)
  =
  \int_{\widebar{X}_0}
  f(t_1,\dotsc,t_\kappa)\,
  s_{1/t_1}(V_0)\dotsm s_{1/t_{\kappa}}(V_0),
\]
and obtain a multivariate Laurent polynomial in \(t_1,\dotsc,t_\kappa\) over \(\Q\).
\item[$\square$\,\scshape Step 1b
  ]Expand the universal rational function \(\Phi_{\kappa}(t_1,\dotsc,t_{\kappa})\) successively with respect to \(t_1,t_2,\dotsc\) up to \(t_{\kappa}\). 
Obtain, not a multivariate Laurent series, but what has been called an iterated Laurent series, similarly denoted by \(\varPhi_\kappa(t_1,\dotsc,t_\kappa)\) -- notice the slanted \(\varPhi\).
\item[$\square$\,\scshape Step 2]
Compute the Cauchy product
\(
  I(t_1,\dotsc,t_\kappa)\,
  \varPhi_{\kappa}(t_1,\dotsc,t_\kappa)
\)
of the multivariate Laurent polynomial \(I(t_1,\dotsc,t_\kappa)\) and of the iterated Laurent series \(\varPhi_{\kappa}(t_1,\dotsc,t_\kappa)\) in the field of iterated Laurent series over \(\Q\).
Lastly, extract the coefficient of the monomial \(t_1^r\dotsm t_\kappa^r\) in the obtained multivariate formal series, and receive the sought element \(I\in \Q\).
\end{description}

\subsection*{Obstacles and forthcoming results}
Really computing \(I\) proves to be quite delicate in practice.
The first effective result in any dimension towards the Green-Griffiths conjecture was obtained in 2010 by Diverio, Merker and Rousseau (\cite{MR2593279}), using step-by-step algebraic elimination, for entire curves \(\C\to X_{d}\subset\P^{n+1}\) with values in generic hypersurfaces \(X_{d}\) of degree \(d\) in \(\P^{n+1}\), with an estimated sufficient lower bound:
\[
  d
  \geq
  2^{n^{5}}.
\]
Some time after, in \cite{arxiv:1011.4710}, Gergely Bérczi made a substantial progress by replacing the elimination step of \cite{MR2593279} by an iterated residue formula, and he reached the lower bound:
\[
  d\geq n^{8n}.
\]

Using our main result, the difficulty is that in general, Step \textsc{1b} does not yield a single iterated Laurent series, but produces an involved product of several iterated Laurent series. 
Then, it is very difficult to determine even the sign of any individual coefficient of \(\varPhi_{\kappa}(t_1,\dotsc,t_\kappa)\),
because this amounts to disentangle the large product:
\[
  \prod_{2\leq i\leq j\leq\kappa}
  \frac{t_{j}-2\,t_{i}}
  {t_{j}-2\,t_{i}+t_{i-1}}
  =
  \prod_{2\leq i\leq j\leq\kappa}
  \left(
  1
  -
  \sum_{p=0}^{\infty}
  \frac{t_{i-1}(2\,t_{i}-t_{i-1})^{p}}
  {t_{j}^{p+1}}
  \right).
\]

On the other hand, it is relatively easy to control the absolute value of these coefficients, using a convergent majorant series with positive coefficients, whence suppressing the problem of signs:
\begin{multline*}
  \abs{\mathrm{coeff}}
  \left(
  \prod_{2\leq i\leq j\leq\kappa}
  \frac{t_{j}-2\,t_{i}}
  {t_{j}-2\,t_{i}+t_{i-1}}
  \right)
\leq
\mathrm{coeff}
\left(
  \prod_{2\leq i\leq j\leq\kappa}
  \frac{t_{j}-2\,t_{i}}
  {t_{j}-2\,t_{i}-t_{i-1}}
  \right)
  =\\=
  \mathrm{coeff}
  \left(
  \prod_{2\leq i\leq j\leq\kappa}
  \left(
  1
  +
  \sum_{p=0}^{\infty}
  \frac{t_{i-1}(2\,t_{i}+t_{i-1})^{p}}
  {t_{j}^{p+1}}
  \right)
  \right)
\end{multline*}
-- notice that \(+t_{i-1}\) in the denominator becomes \(-t_{i-1}\).
And this allows us to use (plainly) the triangle inequality in \cite{arXiv:1402.1396}, to attain an effective lower bound on the degree \(d\) of generic smooth hypersurfaces  \(X_{d}\subset\P^n\) such that all entire curves \(\C\to\P^n\setminus X_{d}\) are algebraically degenerate:
\[
  d
  \geq
  (5n)^2\,n^n,
\]
a lower bound which also holds for curves with values in a generic hypersurface \(X_{d}\subset \P^{n+1}\).

A better understanding of the combinatorics of the series \(\varPhi_{\kappa}\) would allow to improve in a more subtle way this effective lower bound. It seems reasonable to reach an exponential bound:
\[
  d
  \stackrel{?}\geq
  (\mathsf{constant})^n.
\]

\subsection*{Acknowledgments.}
I would like to thank \textsl{Dr. Gergely Bérczi}, for his kind explanation of his article \cite{arxiv:1011.4710} in march 2012 during the annual meeting of the GDR ``GAGC'' at the CIRM (Marseille).
I would like to thank \textsl{Pr. Christophe Mourougane}, who accepted to welcome me in Rennes to discuss in details the present work, and then encouraged me.
I would like to thank \textsl{Dr. Damian Brotbek}, with whom I exchanged a lot. 
I owe him the idea of using Segre classes in order to achieve fiber integration (following Mourougane). 
\smallskip

I am deeply grateful to my thesis advisor \textsl{Pr. Joël Merker}, to whom I would like to dedicate this paper.

\section{Demailly tower of (logarithmic) directed manifolds}
A \textsl{directed manifold} is defined to be a couple \((X,V)\) where \(X\) is a complex manifold, equipped with a (not necessary integrable) holomorphic subbundle \(V\subset T_{X}\) of its holomorphic tangent bundle. There is a natural generalization of this definition in the logarithmic setting. A \textsl{log-directed manifold} is by definition a triple \((\widebar{X}, D, V)\) where \((\widebar{X}, D)\) is a log-manifold and the distribution \(V\subset T_{\widebar{X}}(-\log D)\) is a (not necessary integrable) subbundle of the logarithmic tangent bundle. 

Given a log-directed manifold \((\widebar{X},D,V)\),  following Dethloff and Lu~\cite{MR1824906}, we construct  the \textsl{Demailly tower of projectivized bundles} \((\widebar{X}_i,D_i,V_i)\) on \(\widebar{X}\) by induction on \(i\geq 0\). 
This construction is formally the same as the construction~\cite{MR1492539} of the Demailly tower in the so-called \textsl{compact case}, \textit{i.e.} where there is no divisor \(D\). 
The only slight modification to keep in mind in the genuine logarithmic setting is that \(V\) is a holomorphic subbundle of the logarithmic tangent bundle \(T_{\widebar{X}}\bigl(-\log D\bigr)\). 

\subsection*{Projectivization of a log directed manifold (\cite{MR1492539,MR1824906})}
Recall that for a vector bundle \(E\to \widebar{X}\) on a smooth manifold \(\widebar{X}\)
with projective bundle of lines 
\(\pi\colon P(E)\to \widebar{X}\)
, the \textsl{tautological line bundle}:
\[
  \mathcal{O}_{P(E)}(-1)
  \to
  P(E)
\]
is defined as the subbundle of the pullback bundle 
\(
\pi^\star E 
\to
P(E)
\)
with trivial fiber.

The dual line bundle \(\mathcal{O}_{P(E)}(1)\bydef\mathcal{O}_{P(E)}(-1)^{\vee}\) fits into the following \textsl{Euler exact sequence} (\cite[B.5.8]{MR1644323}):
\[
  0
  \to
  \mathcal{O}_{P(E)}
  \to
  \pi^\star E\otimes\mathcal{O}_{P(E)}(1)
  \to
  T_\pi
  \to
  0,
\]
where \(T_\pi\bydef\ker(\pi_\star)\) stands for the \textsl{relative tangent bundle} of \(P(E)\) over \(\widebar{X}\), that itself fits into the following short exact sequence:
\[
  0
  \to
  T_{\pi}
  \hookrightarrow
  T_{P(E)}
  \stackrel{\pi_{\star}}
  \longrightarrow
  \pi^{\star}T_{\widebar{X}}
  \to
  0.
\]

\medskip

We now recall the inductive step \(\bigl(\widebar{X}',V',D'\bigr)\stackrel{\pi}\to\bigl(\widebar{X},V,D\bigr)\) of the construction of the Demailly tower.
Keep in mind that \(V\) is a subbundle of \(T_{\widebar{X}}(-\log D)\) and that \(V'\) has to be a subbundle of the logarithmic tangent bundle \(T_{\widebar{X}'}(-\log D')\).

Firstly, for \(\widebar{X}'\) take the total space \(P(V)\) of the projective bundle of lines of \(V\):
\[
  \widebar{X}'
  \bydef
  P(V)
  \stackrel{\pi}\longrightarrow
  \widebar{X},
\]
and in order to make \(\pi\) a log-morphism set:
\[
D'
  \bydef
  \pi^{\moinsun}(D)
  \subset
  \widebar{X}'.
\]

Next, by definition of the relative tangent bundle \(T_{\pi}\bydef\ker(\pi_{\star})\) of the log-morphism \(\pi\) one has the following short exact sequence:
\[
  0
  \to
  T_{\pi}
  \hookrightarrow
  T_{\widebar{X}'}(-\log D')
  \stackrel{\pi_{\star}}
  \longrightarrow
  \pi^{\star}
  T_{\widebar{X}}(-\log D)
  \to
  0,
\]
and since by assumption \(V\subset T_{\widebar{X}}(-\log D)\), the tautological line bundle of \(\widebar{X}'=P(V)\) is a subbundle of the bundle in the right-hand slot:
\[
  \mathcal{O}_{\widebar{X}'}(-1)
  \subset
  \pi^{\star}V
  \subset
  \pi^{\star}T_{\widebar{X}}(-\log D),
\]
whence one can define a subbundle \(V'\subset T_{\widebar{X}'}(-\log D')\) by taking:
\[
  V'
  \bydef
  (\pi_\star)^{\moinsun}
  \mathcal{O}_{\widebar{X}'}(-1).
\]
The only thing to verify in order to get a log directed manifold is that \(V'\) is a holomorphic subbundle of \(T_{\widebar{X}'}(-\log D')\). Since \((\pi_{\star})^{\moinsun}\) has maximal rank everywhere, as it is a bundle projection, this is the case~(\cite{MR1824906}).

By its very definition \(V'\) then fits into the following short exact sequence:
\[
  0
  \to
  T_\pi
  \hookrightarrow
  V'
  \stackrel{\pi_\star}
  \longrightarrow
  \mathcal{O}_{\widebar{X}'}(-1)
  \to
  0.
\]
and thus the rank of \(V'\) is the same as the rank of \(V\), because:
\[
  \rk(V')
  =
  \rk(P(V))
  +
  1
  =
  \rk(V)-1+1.
\]
Starting from a bundle \(V_0\) having rank \(r+1\),
by iterating the construction \(\kappa\) times, one thus gets a tower of projectivized bundles 
\[
  \bigl(\widebar{X}_{\kappa},V_{\kappa},D_{\kappa}\bigr)
  \stackrel{\pi_{\kappa}}\longrightarrow
  \dotsb
  \stackrel{\pi_{2}}\longrightarrow
  \bigl(\widebar{X}_1,V_1,D_1\bigr)
  \stackrel{\pi_{1}}\longrightarrow
  \bigl(\widebar{X}_0,V_0,D_0\bigr)
\]
with
\(\rk V_{i}=r+1\) and
\(
  n_{i}
  \bydef
  \dim\,\widebar{X}_{i}
  =
  \dim(\widebar{X}_0)
  +
  i\;\bigl(\mathrm{rk}\,P(V_0)\bigr)
  =
  n+i\,r
\).

\subsection*{Existence of global jet differentials}
The fibers of the Demailly-Semple bundle of jet differentials \(E_{\kappa,m}(V_{0})^{\star}(\log D_{0})\) carries much complexity (\cite{MR2477967,MR2679388}). In order to prove the existence of global jet differentials of order \(\kappa=\dim(\widebar{X})\),
one is led to consider a much more tractable line bundle, constructed in \cite[6.13]{MR1492539} as a linear combination with non negative integer coefficients \((a_1,\dotsc,a_{\kappa})\): 
\[
  \mathcal{O}_{\widebar{X}_k}(a_1,a_2,\dotsc,a_k)
  \bydef
  (\pi_{k,1})^\star\mathcal{O}_{\widebar{X}_1}(a_1)
  \otimes
  (\pi_{k,2})^\star\mathcal{O}_{\widebar{X}_2}(a_2)
  \otimes
  \dotsb
  \otimes
  \mathcal{O}_{\widebar{X}_k}(a_k).
\]
If \(a_1+\dotsc+a_\kappa=m\), the direct image
\(
  (\pi_{\kappa,0})_\star 
  \mathcal{O}_{\widebar{X}_{\kappa}}(a_1,\dotsc,a_{\kappa})
\) 
may be seen as a subbundle of the Demailly-Semple bundle of jet differentials (\cite{MR1492539,MR1824906}).

For a suitable choice of the parameters \(a_1,\dotsc,a_{\kappa}\), the line bundle \(\mathcal{O}_{\widebar{X}_{\kappa}}(a_1,\dotsc,a_{\kappa})\) has some positivity properties, that can be used together with the following Demailly-Trapani \textsl{algebraic Morse inequalities}~(\cite{MR1339712,MR1492539}) in order to establish the existence of global jet differentials.
\begin{Thm}[Weak algebraic Morse inequalities] 
  \label{thm_morse}
  For any holomorphic line bundle \(L\) on a \(N\)-dimensional compact manifold \(\widebar{X}\), that  can be written as the difference \(L=F\otimes G^\vee\) of two nef line bundles \(F\) and \(G\), one has:
  \[
    h^0\bigl(\widebar{X},L^{\otimes k}\bigr)
    \geq
    k^N\;
    \frac{
      \bigl(F^N\bigr)-k\,\bigl(F^{N-1}\cdot G\bigr) 
    }
    {N!}
    - 
    o(k^N).
  \]
\end{Thm}

For a choice of \(a_1,\dotsc,a_{\kappa}\in\N^{\kappa}\) such that:
\begin{equation}
  \label{eq:relativelyample}
  a_{\kappa-1}>2\,a_{\kappa}>0
  \quad;\quad
  a_i\geq3\,a_{i+1}\quad(i=1,\dotsc,\kappa-2),
\end{equation}
the line bundle \(\mathcal{O}_{\widebar{X}_{\kappa}}(a_1,\dotsc,a_{\kappa})\) is relatively ample along the fibers of \(\widebar{X}_{\kappa}\) over \(\widebar{X}\)~(\textit{cf}. \cite{MR1492539}).
It hence suffices to multiply it by a sufficiently positive power \(\pi_{\kappa,0}^{\star}A^{\otimes l}\) of a given ample line bundle \(A\to\widebar{X}\), in order to get an ample (hence nef) line bundle (\textit{cf}. \cite{MR2095471}). 
On the other hand \(\pi_{\kappa,0}^{\star}A^{\otimes l+1}\) is nef for it is the pullback of a nef line bundle.

It gives an expression of the line bundle \(\mathcal{O}_{\widebar{X}_{\kappa}}(a_1,\dotsc,a_{\kappa})\otimes(\pi_{\kappa,0})^{\star}A^{\vee}\) as the difference \(F\otimes G^{\vee}\) of two nef line bundles:
\[
  F
  \bydef
  \mathcal{O}_{\widebar{X}_{\kappa}}(a_1,\dotsc,a_{\kappa})
  \otimes
  (\pi_{\kappa,0})^{\star}
  A^{\otimes l}
  \quad\text{and}\quad
  G
  \bydef
  (\pi_{\kappa,0})^{\star}
  A^{\otimes l+1}
\]
In order to prove the existence of global sections:
\[
  P
  \in
  H^0
  \Bigl(
  \widebar{X}_{\kappa},
  \mathcal{O}_{\widebar{X}_{\kappa}}(a_1,\dotsc,a_{\kappa})
  \otimes
  (\pi_{\kappa,0})^\star A^\vee
  \Bigr),
\] 
it hence remains to show the positivity of the following intersection number: 
\[
  I
  \bydef
  \int_{\widebar{X}_{\kappa}}
  c_1(F)^{n_{\kappa}}
  -
  n_{\kappa}\,
  c_1(F)^{n_{\kappa}-1}c_1(G)
  \quad{\scriptstyle(n_\kappa=\dim\,\widebar{X}_{\kappa})}.
\]

We will give a formula for computing such an intersection product.

\section{Fiber Integration on the Demailly tower}
It is convenient to bring down the computation to the basis and we will now provide a formula for this purpose.
Noteworthy, the proof of this formula involves iterated Laurent series, in the same spirit as the residue formula of Berczi~\cite{arxiv:1011.4710,berczi2012}. 
%Note that we can similarly prove the residue formula appearing in \cite{Berczi2010}, by setting instead \(L_i=\mathcal{O}_{X_i}(-1)\).
However, we will not use equivariant geometry like this author, but only basic lemmas of intersection theory, more precisely some of the properties of Segre classes exposed in the book of Fulton~\cite[Chap. 3]{MR1644323}.
We now first briefly recall these properties.

\subsection*{Segre classes on the Demailly tower}
To go down one level, from \(\widebar{X}_{i+1}=P(V_i)\) to \(\widebar{X}_i\), we will use the very definition of the \(j\)-th \textsl{Segre class} of a vector bundle \(E\to X\) (having rank \(r+1\)), namely the fiber integration formula:
\begin{equation}
  \label{eq:segre}
  \int_{X} s_j(E)\,\alpha
  =
  \int_{P(E)}
  u^{j+r}\,p^\star\alpha
  \qquad
  {\scriptstyle (j\geq0,\ \alpha\in H^{\smbullet}X)},
\end{equation}
where:
\[
  p\colon P(E)\to X
  \quad
  \text{and}
  \quad
  u
  \bydef
  c_1\bigl(\mathcal{O}_{P(E)}(-1)^{\vee}\bigr).
\]

We want to apply this formula in order to eliminate the powers of the first Chern classes \(v_1,\dotsc,v_{\kappa}\) of the vector bundles \(L_{i}\to \widebar{X}_{i}\). We will proceed by induction.

\medskip

It is well known that the total Segre class of a vector bundle is the same as the inverse of its total Chern class.
Thus, the total Segre class enjoys the Whitney formula.

Because we will obtain a result that is independent of the geometric context, we will deliberately be ambiguous about it.
The only property of the Demailly construction that we use in what follows is the existence of the two short exact sequences:
\[
  0 \to T_\pi \to V' \to \mathcal{O}_{\widebar{X}'}(-1) \to 0
\]
and:
\[
  0 \to \mathcal{O}_{\widebar{X}'} \to \pi^\star V\otimes \mathcal{O}_{\widebar{X}'}(1) \to T_\pi \to 0,
\]
where \((\widebar{X}',V')\stackrel{\pi}\to(\widebar{X},V)\) is the inductive step of the Demailly construction.

Now, consider the following observation: the twist by a line bundle does not change the projective bundle of lines of \(V'\), but only the transition functions. 
Moreover, one can twist short exact sequences by line bundles. 

We can hence chose a line bundle \(L'\) on \(\widebar{X}'\) that makes the induction more easy.
We will twist both short exact sequences by the same line bundle, because we do not want \(T_\pi\) to appear in the final formula below.
Also, we do not want anymore the central term of the second short exact sequence to be a product of line bundles with different base spaces \(\widebar{X}\) and \(\widebar{X}'\) but rather want it to be the pullback of a single bundle on the lower level \(\widebar{X}\), that is:
\[
  \pi^\star V\otimes\mathcal{O}_{\widebar{X}'}(1)
  \otimes
  L'
  =
  \pi^\star\bigl(
  V\otimes L
  \bigr),
\]
for a certain \(L\to\widebar{X}\) (in practice given by the preceding induction steps).
Consequently, we have to take:
\begin{equation}
  \label{eq:ll'}
  L'
  \bydef
  \mathcal{O}_{\widebar{X}'}(-1)
  \otimes
  \pi^\star L
  =
  \mathcal{O}_{P(V\otimes L)}(-1).
\end{equation}
Notice that accordingly the term \(\mathcal{O}_{\widebar{X}'}(-1)\) can now be replaced by \(\bigl(L'\otimes\pi^\star L^{\vee}\bigr)\) in the first exact sequence.

Once twisted by \(L'\), the above two short exact sequences become:
\[
  \left\{ \begin{array}{c!{\to}c!{\to}c!{\to}c!{\to}c}%}
    0 & T_\pi\otimes L' & V'\otimes L' & \bigl(L'\otimes\pi^\star L^{\vee}\bigr)\otimes L' & 0\\
    0 & L' & \pi^\star\bigl(V\otimes L\bigr) & T_\pi\otimes L' & 0
  \end{array} \right..
\]
By the Whitney formula, the first line yields:
\[
  s(V'\otimes L')
  =
  s(T_\pi\otimes L')\,
  s\bigl((L')^{\otimes2}\otimes\pi^\star L^{\vee}\bigr),
\]
while the second line yields:
\[
  \pi^\star
  s\bigl(V\otimes L\bigr)
  =
  s\bigl(T_\pi\otimes L'\bigr)\,
  s\bigl(L'\bigr).
\]
Thus, we can eliminate \(T_\pi\), as it was our intention, in order to get the induction formula:
\[
  s\bigl(V'\otimes L'\bigr)
  =
  \frac
  {s\bigl((L')^{\otimes2}\otimes\pi^\star L^{\vee}\bigr)}
  {s\bigl(L'\bigr)}\;
  \pi^\star
  s\bigl(V\otimes L\bigr).
\]

Now, for a line bundle \(L\to X\), the total Segre class is the finite sum: 
\[
  s_{\smbullet}(L)
  =
  \bigl(1-c_1(L^\vee)\bigr)^{\moinsun}
  =
  1+c_1(L^{\vee})+c_1(L^\vee)^2+\dotsb+c_{1}(L^{\vee})^{\dim(X)}
\]
-- we use the first Chern class of the dual in order to have positive signs.
Let
\(
v\bydef c_1\bigl(L^{\vee}\bigr)
\)
and
\(
v'\bydef c_1\bigl((L')^{\vee}\bigr).
\)
We get the induction formula:
\begin{equation}
  \label{eq:ll'segre}
  s\bigl(V'\otimes L'\bigr)
  =
  \varphi\bigl(v',v\bigr)\;
  \pi^\star
  s\bigl(V\otimes L\bigr),
\end{equation}
where \(\varphi(x,y)\) is the truncated double Taylor expansion of the rational function 
\((1-x)\,\bigl(1-2x+y\bigr)^{-1}\):
\[
  \varphi\bigl(x,y\bigr)
  \bydef
  (1-x)\,
  \sum_{k=0}^{n_{\kappa-1}}(2\,x-y)^k.
\]

\smallskip

Considering \eqref{eq:ll'}, we construct the \textit{ad hoc} sequence of line bundles \(L_i\to \widebar{X}_i\) 
by taking first the tautological line bundle 
\(
  L_1
  \bydef 
  \mathcal{O}_{\widebar{X}_1}(-1)
\)
of \(V_0\)
and then the tautological line bundle of the twisted vector bundle \(V_{i-1}\otimes L_{i-1}\):
\[
  L_i
  \bydef
  \mathcal{O}_{\widebar{X}_i}(-1)
  \otimes
  (\pi_i)^\star L_{i-1}
  =
  \mathcal{O}_{P(V_{i-1}\otimes L_{i-1})}(-1)
  \qquad{\scriptstyle (i=2,\dotsc,\kappa)}.
\]
We will denote by \(v_i\) the first Chern class of the dual of this line bundle:
\[
  v_i
  \bydef
  c_1\bigl(L_i^\vee\bigr)
  =
  c_1\bigl(\mathcal{O}_{P(V_{i-1}\otimes L_{i-1})}(1)\bigr).
\]
Then, by \eqref{eq:ll'segre}, on has the following inductive formulas,
where, for simplicity, we omit the pullbacks:
\[
  s\bigl(V_1\otimes L_1\bigr)
  =
  \varphi(v_{i},0)\,
  s\bigl(V_{0}\bigr)
  \quad\text{and}
  \quad
  s\bigl(V_i\otimes L_i\bigr)
  =
  \varphi(v_{i},v_{i-1})\,
  s\bigl(V_{i-1}\otimes L_{i-1}\bigr),
  \qquad
  {\scriptstyle(i=2,\dotsc,\kappa)}.
\]

Notice that we can reformulate the positivity property \eqref{eq:relativelyample} by using the more explicit expression of the line bundles \(L_i\):
\[
  L_i
  =
  \mathcal{O}_{\widebar{X}_i}(-1)
  \otimes
  \dotsb
  \otimes
  (\pi_{i,2})^\star\mathcal{O}_{\widebar{X}_2}(-1)
  \otimes
  (\pi_{i,1})^\star\mathcal{O}_{\widebar{X}_1}(-1)
  \qquad{\scriptstyle (i=1,\dotsc,\kappa)}.
\]
and the inversion of these formulas
\[
  \mathcal{O}_{\widebar{X}_i}(1)
  =
  L_i^\vee
  \otimes
  L_{i-1};
\] 
in analogy with
\(
  \mathcal{O}_{\widebar{X}_{\kappa}}(a_1,\dotsc,a_{\kappa})
  =
  a_1\,\mathcal{O}_{\widebar{X}_1}(-1)^\vee+\dotsb+a_{\kappa}\,\mathcal{O}_{\widebar{X}_{\kappa}}(-1)^\vee
\) 
consider, a linear combination:
\[
  L(a_1,\dotsc,a_{\kappa})
  \bydef
  a_1\,L_1^\vee+\dotsb+a_{\kappa}\,L_{\kappa}^\vee
\] 
of the line bundles \(L_i^\vee\),
with non negative coefficients \(a_i\),
such that:
\[
  a_1+2\,a_2+\dotsb+\kappa\,a_\kappa=m\in\N,
\]
then, the line bundle \(\pi_{\kappa,0}^\star L(a_1,\dotsc,a_{\kappa})\) may be seen as a certain subbundle of \(E_{\kappa,m}V_0^\star(\log D_0)\) and if:
\[
  a_{\kappa-1}>a_{\kappa}\geq 1
  \quad;\quad
  a_i\geq 2(a_{i+1}+\dotsb+a_{\kappa})\quad
  {\scriptstyle(i\leq\kappa-2)}.
\]
the line bundle \(L(a_1,\dotsc,a_{\kappa})\) is relatively ample along the fibers of \(\widebar{X}_{\kappa}\to \widebar{X}_0\).

\begin{Prop}
  \label{eq:fiberint}
  For any polynomial in the first Chern classes \(v_1,\dotsc,v_{i+1}\) having coefficients in (the pullback of) the cohomology of \(\widebar{X}_0\):
  \[
    f
    \in 
    H^{\smbullet}
    \bigr(\widebar{X}_0\bigl)
    [v_1,\dotsc,v_i,v_{i+1}],
  \]
  the following formula of integration along the fibers of \(\widebar{X}_{i+1}\to\widebar{X}_i\) holds:
  \[
    \int_{\widebar{X}_{i+1}}
    f(v_1,\dotsc,v_i,v_{i+1})
    =
    \bigl[t_{i+1}^{r}\bigr]
    \left(
    \int_{\widebar{X}_i}
    f(v_1,\dotsc,v_i,t_{i+1})\,
    s_{1/t_{i+1}}\bigl(V_i\otimes L_i\bigr)
    \right).
  \]
\end{Prop}
\begin{proof}
  Firstly, by the Leray-Hirsch theorem:
  \[
    H^{\smbullet}
    \bigr(\widebar{X}_0\bigl)
    [v_1,\dotsc,v_i,v_{i+1}]
    =
    H^{\smbullet}
    \bigr(\widebar{X}_{i+1}\bigl)
  \]
  Thus, \(f\) has values in the cohomology ring of \(\widebar{X}_{i+1}\).

  The polynomial \(f\) is of the form: 
  \[
    f(v_1,\dotsc,v_i,v_{i+1})
    =
    \sum_{j=0}^{n_{i+1}}
    (v_{i+1})^j\,
    (\pi_{i+1})^{\star}f_{j}(v_1,\dotsc,v_i).
  \]
  By linearity, the formula will hold for any such sum, if it holds for every monomial:
  \[
    v_{i+1}^j\,
    (\pi_{i+1})^{\star}f_{j}(v_1,\dotsc,v_i).
  \]

  Recall that the line bundles \(L_i\) are constructed by the inductive formula:
  \[
    L_{i+1}
    =
    \mathcal{O}_{P(V_i)}(-1)
    \otimes
    \pi_{i+1}^\star L_i
    =
    \mathcal{O}_{P(V_i\otimes L_i)}(-1).
  \]
  Thus, it can be thought of as the tautological line bundle of the projective bundle:
  \[
    P(V_i\otimes L_i)
    \simeq
    P(V_i)
    \defby
    \widebar{X}_{i+1}.
  \]
  Then the above fiber integration formula \eqref{eq:segre} yields at once:
  \[
    \tag{$*$}
    \label{eq:fiber}
    \int_{\widebar{X}_{i+1}}
    v_{i+1}^j\,
    (\pi_{i+1})^{\star}f_{j}(v_1,\dotsc,v_i)
    =
    \int_{\widebar{X}_i} 
    s_{j-r}
    \bigl(V_i\otimes L_i\bigr)\;
    f_{j}(v_1,\dotsc,v_i).
  \]
  In particular, this integral is zero for indices \(j\) smaller than \(r\).

  The problem is now to obtain the individual Segre classes from the total Segre class.
  In that aim, we will use the formalism of generating functions. 
  Recall that, in analogy with Chern polynomial, for a vector bundle \(E\to X\) over a \(N\) dimensional manifold \(X\), we have introduced the generating function \(s_u(E)\) of the Segre classes of \(E\), that is:
  \[
    s_u(E)
    \bydef
    s_0(E)+u\,s_1(E)+u^2\,s_2(E)+\dotsb+u^N\,s_N(E).
  \]
  Then, by taking \(t=1/u\), we obtain a Laurent polynomial:
  \[
    s_{1/t}(E)
    \bydef
    \frac{s_0(E)}{t^0}+\frac{s_1(E)}{t^1}+\frac{s_2(E)}{t^2}+\dotsb+\frac{s_N(E)}{t^N},
  \]
  in which the \(\bigl(j-r\bigr)\)-th Segre class involved in the fiber integration appears as the coefficient:
  \[
    s_{j-r}(E)
    =
    \bigl[1/t^{j-r}\bigr]
    s_{1/t}(E)
    =
    \bigl[t^{r}\bigr]
    \Bigl(t^j\,s_{1/t}(E)\Bigr).
  \]

  Therefore, by replacing in the integration formula \eqref{eq:fiber}:
    \begin{align*}
      \int_{\widebar{X}_{i+1}}
      v_{i+1}^j\,
      (\pi_{i+1})^{\star}f_{j}(v_1,\dotsc,v_i)
      &=
      \int_{\widebar{X}_i} 
      \bigl[{t_{i+1}}^{r}\bigr]
      \Bigl(
      t_{i+1}^j\;
      s_{1/t_{i+1}}\bigl(V_i\otimes L_i\bigr)
      \Bigr)\,
      f_{j}(v_1,\dotsc,v_i).
      \\
      &=
      \bigl[{t_{i+1}}^{r}\bigr]
      \left(
      \int_{\widebar{X}_i}
      t_{i+1}^j\;
      f_j(v_i,\dotsc,v_1)\;
      s_{1/t_{i+1}}\bigl(V_i\otimes L_i\bigr)
      \right).
    \end{align*}

  Notice that inside of the parenthesis there is the product of a monomial by a Laurent polynomial. Thus, only a finite number of terms are involved and there is no objection to switching the integral and the coefficient extraction.

  The obtained formula is exactly the sought formula for the considered monomial:
  \[
    v_{i+1}^j\,
    (\pi_{i+1})^{\star}f_{j}(v_1,\dotsc,v_i),
  \]
  and this ends the proof.
\end{proof}

\subsection*{Iteration of the integration formula}
In order to iterate the fiber integration, we introduce the following formalism: for \(i=0,1,\dotsc,\kappa\), we denote by \(\underline{vt}_{\,i}\) the \(\kappa\)-tuple obtained from:
\[
  \underline{v}
  \bydef
  (v_1,\dotsc,v_{\kappa})
\]
by replacing the last \(i\) components \(v_{\kappa-i+1}\), \dots, \(v_{\kappa}\) by the corresponding parameters \(t_{\kappa-i+1}\), \dots, \(t_{\kappa}\), \textit{i.e.}:
\[
  \underline{vt}_{\,i}
  \bydef
  (v_1,\dotsc,v_{\kappa-i},t_{\kappa-i+1},\dotsc,t_\kappa)
  \qquad
  {\scriptstyle (i=0,1,\dotsc,\kappa)}.
\]
With this notation the fiber integration formula \eqref{eq:fiberint} just above yields directly, that
for any polynomial in the first Chern classes \(v_1,\dotsc,v_{i+1}\) having coefficients in (the pullback of) the cohomology of \(\widebar{X}_0\), being a Laurent polynomial in the formal parameters \(t_{i+2},\dotsc,t_{\kappa}\):
\[
  f
  \in 
  H^{\smbullet}
  \bigr(\widebar{X}_0\bigl)
  [v_1,\dotsc,v_i,v_{i+1}]
  [t_{i+2},t_{i+2}^{\moinsun},\dotsc,t_{\kappa},t_{\kappa}^{\moinsun}],
\]
the following formula of integration along the fibers of \(\widebar{X}_{i+1}\to\widebar{X}_i\) holds:
\begin{equation}
  \label{eq:fibInt}
  \int_{\widebar{X}_{i}}
  f\bigl(\underline{vt}_{\,\kappa-i}\bigr)
  =
  \bigl[t_{i}^{r}\bigr]
  \int_{\widebar{X}_{i-1}}
  f\bigl(\underline{vt}_{\,\kappa-i+1}\bigr)\;
  s_{1/t_{i}}
  \bigl(V_{i-1}\otimes L_{i-1}\bigr).
\end{equation}

Notice that in the above formula the form of the polynomial appearing in the integrand:
\[
  f\bigl(\underline{vt}_{\,\kappa-i+1}\bigr)\;
  s_{1/t_{i}}
  \bigl(V_{i-1}\otimes L_{i-1}\bigr)
  \in 
  H^{\smbullet}
  \bigr(\widebar{X}_0\bigl)
  [v_1,\dotsc,v_i]
  [t_{i+1},t_{i+1}^{\moinsun},\dotsc,t_{\kappa},t_{\kappa}^{\moinsun}],
\]
allows to iterate this formula in order to integrate along the fibers of \(\widebar{X}_{i-1}\to \widebar{X}_{i-2}\).
For short, we denote the appearing polynomial rings by:
\[
  \Lambda[\underline{vt}_{\kappa-i}]
  \bydef
  H^{\smbullet}
  \bigr(\widebar{X}_0\bigl)
  [v_1,\dotsc,v_i]
  [t_{i+1},t_{i+1}^{\moinsun},\dotsc,t_{\kappa},t_{\kappa}^{\moinsun}]
  \quad{\scriptstyle(i=0,1,\dotsc,\kappa)}.
\]
One has thus:
\[
  \Lambda[\underline{vt}_0]=
  H^{\smbullet}\bigl(\widebar{X}_{0}\bigr)[\underline{v}]=
  H^{\smbullet}\bigl(\widebar{X}_{\kappa}\bigr)
  \quad\text{and}\quad
  \Lambda[\underline{vt}_\kappa]=
  H^{\smbullet}\bigl(\widebar{X}_{0}\bigr)[\underline{t},\underline{t}^{\moinsun}].
\]

We have first to investigate the dependence with respect to \(v_i\) of the appearing power series \(s_{1/t_{i+1}}\bigl(V_i\otimes L_i\bigr) \).
The induction formula \eqref{eq:segre} precisely provides us with this information.
Thanks to it, we can split the power series \(s_{1/t}\bigl(V_i\otimes L_i\bigr)\) in two parts:
\[
  s_{1/t_{j}}\bigl(V_{i}\otimes L_{i}\bigr)
  =
  \underbrace{
    \varphi
    \bigg(\frac{v_{i}}{t_{j}},\frac{v_{i-1}}{t_{j}}\bigg)
  }
  _{\in\Lambda[\underline{vt}_{\kappa-i}]}
  \;
  \underbrace{
    \vphantom{\bigg(}
    s_{1/t_{j}}\bigl(V_{i-1}\otimes L_{i-1}\bigr)
  }
  _{\in\Lambda[\underline{vt}_{\kappa-i+1}]},
\]
or for \(i=1\):
\[
  s_{1/t_{j}}\bigl(V_{1}\otimes L_{1}\bigr)
  =
  \underbrace{
    \varphi
    \bigg(\frac{v_{1}}{t_{j}},0\bigg)
  }
  _{\in\Lambda[\underline{vt}_{\kappa-1}]}
  \;
  \underbrace{
    \vphantom{\bigg(}
    s_{1/t_{j}}\bigl(V_{0}\otimes L_{0}\bigr)
  }
  _{\in H^{\smbullet}(\widebar{X}_{0})[\underline{t},\underline{t}^{\moinsun}]},
\]
the first of which depends on \(v_i\) whereas the second does not.

Write for short:
\[
  \varPhi_{k,l}(t_1,\dotsc,t_\kappa)
  \bydef
  \varphi
  \bigg(\frac{t_k}{t_l},\frac{t_{k-1}}{t_l}\biggr)
  \qquad
  {\scriptstyle (k=2,\dotsc,\kappa-1,\ k+1\leq l\leq\kappa}),
\]
and:
\[
  \varPhi_{1,l}(t_1,\dotsc,t_\kappa)
  \bydef
  \varphi
  \bigg(\frac{t_1}{t_{l}},0\biggr)
  \qquad
  {\scriptstyle (2\leq l\leq\kappa)},
\]
in such way that, for any two positive integers \(k<l\):
\begin{equation}
  \label{eq:phikl}
  s_{1/t_l}\bigl(V_{k}\otimes L_{k}\bigr)
  =
  \varPhi_{k,l}\bigl(\underline{vt}_{k}\bigr)\,
  s_{1/t_l}\bigl(V_{k-1}\otimes L_{k-1}\bigr)
  \quad
  {\scriptstyle(1\leq k<l\leq\kappa)}.
\end{equation}

Let \(\varPhi_i\) be the product of the terms in the \(i\) last lines of the array:
\[
  \xym[0]{6,6}{
        \ar@{.}[5,5]1\ar@{.}[5,0]&\varPhi_{1,2}\ar@{.}[0,4]&&&&\varPhi_{1,\kappa}\\
        &&\varPhi_{2,3}\ar@{.}[dddrrr]\ar@{.}[rrr]&&&\varPhi_{2,\kappa}\ar@{.}[ddd]\\
        &&&&&\\
        &&&&&\\
        &&& &&\varPhi_{\kappa-1,\kappa}\\
        1\ar@{.}[0,5]&&&&&\hskip10pt 1\hskip10pt
      },
\]
that is the product of \(\bigl(i(i-1)/2\bigr)\) terms:
\[
  \varPhi_i(t_1,\dotsc,t_\kappa)
  \bydef
  \prod_{\kappa-i+1\leq k<l\leq\kappa}
  \varPhi_{k,l}(t_1,\dotsc,t_\kappa).
\]
As an example, 
\(
  \varPhi_1
  (t_1,\dotsc,t_\kappa)
  =
  1.
\)

The following lemma will be used in order to isolate the variable \(v_{\kappa-i}\), that is to eliminate after the \(i\)-th step of the fiber integration. 
\begin{Lem}
  [Isolation of \(v_{\kappa-i}\)]
  \label{lem:phi}
  For any \(i=0,1,\dotsc,\kappa-1\) one has the following relation between \(\varPhi_i\) and \(\varPhi_{i+1}\):
  \[
      \varPhi_{i}\bigl(\underline{vt}_{\,i}\bigr)
      \!\!\!
      \prod_{j=\kappa-i+1}^{\kappa}
      \!\!\!
      s_{1/t_j}\bigl(V_{\kappa-i}\otimes L_{\kappa-i}\bigr)
      =
      \varPhi_{i+1}\bigl(\underline{vt}_{\,i}\bigr)
      \!\!\!
      \prod_{j=\kappa-i+1}^{\kappa}
      \!\!\!
      s_{1/t_j}\bigl(V_{\kappa-(i+1)}\otimes L_{\kappa-(i+1)}\bigr).
  \]
\end{Lem}
\begin{proof}
  Recall the induction formula displayed above:
  \[
    \frac
    {s_{1/t_j}\bigl(V_i\otimes L_i\bigr)}
    {s_{1/t_j}\bigl(V_{i-1}\otimes L_{i-1}\bigr)}
    =
    \varPhi_{i,j}
    \bigl(\underline{vt}_{\,\kappa-i}\bigr).
  \]
  Thus, one has:
  \[
    \frac
    {\prod_{j=\kappa-i+1}^{\kappa}
    s_{1/t_j}\bigl(V_{\kappa-i}\otimes L_{\kappa-i}\bigr)}
    {\prod_{j=\kappa-i+1}^{\kappa}
    s_{1/t_j}\bigl(V_{\kappa-(i+1)}\otimes L_{\kappa-(i+1)}\bigr)}
    =
    \prod_{j=\kappa-i+1}^{\kappa}
    \varPhi_{\kappa-i,j}\bigl(\underline{vt}_{\,\kappa-i}\bigr).
  \]
  Now, by definition of \(\varPhi_i\) and \(\varPhi_{i+1}\):
  \[
    \frac
    {\varPhi_{i+1}}
    {\varPhi_{i}}
    \bigl(\underline{vt}_{\,\kappa-i}\bigr)
    =
    \frac
    {\prod_{\kappa-i\leq k<l\leq\kappa}
    \varPhi_{k,l}}
    {\prod_{\kappa-i+1\leq k<l\leq\kappa}
    \varPhi_{k,l}}
    \bigl(\underline{vt}_{\,\kappa-i}\bigr)
    =
    \prod_{l=\kappa-i+1}^{\kappa}
    \varPhi_{\kappa-i,l}
    \bigl(\underline{vt}_{\,\kappa-i}\bigr).
  \]
  Hence, we get the announced result.
\end{proof}
Notice that in the right hand side of the obtained formula, only the first factor depends on \(v_{\kappa-i}\).

This result is given by anticipation of the proof of main theorem \eqref{thm:fibInt1}. However we can already notice that, \textit{e.g.}:
\[
  \varPhi_{1}\bigl(\underline{vt}_{\,1}\bigr)
  \prod_{j=\kappa-1+1}^{\kappa}
  s_{1/t_j}\bigl(V_{\kappa-i}\otimes L_{\kappa-i}\bigr)
  =
  s_{1/t_\kappa}\bigl(V_{\kappa-1}\otimes L_{\kappa-1}\bigr),
\]
is the term appearing in the first step of the fiber integration.

\begin{Thm}[Fiber Integration on the Demailly tower]
  \label{thm:fibInt1}
  Any polynomial:
  \[
    f\in H^{\bullet}(\widebar{X}_0,V_0)[t_1,\dotsc,t_{\kappa}],
  \]
  in \(\kappa\) variables \(t_1,\dotsc,t_{\kappa}\), with coefficients in the cohomology ring \(H^{\bullet}(\widebar{X}_0,V_0)\), yields a cohomology class:
  \[
    f\bigl(\underline{v}\bigr)
    =
    f\bigl(v_1,\dotsc,v_\kappa\bigr)
    \in
    H^\smbullet\bigl(\widebar{X}_\kappa\bigr),
  \]
  that can be integrated along the fibers of the projective bundle \(\widebar{X}_\kappa\to \widebar{X}_0\) according to the formula:
  \[
    \int_{\widebar{X}_{\kappa}}
    f\bigl(\underline{v}\bigr)
    =
    \bigl[
      t_1^{r}\dotsm t_{\kappa}^{r}
    \bigr]
    \biggl(
    \varPhi_{\kappa}\bigl(\underline{t}\bigr)
    \int_{\widebar{X}_0}
    f\bigl(\underline{t}\bigr)\;
    s_{1/t_1}(V_0)
    \dotsm
    s_{1/t_{\kappa}}(V_0)
    \biggr).
  \]
\end{Thm}
\begin{proof}
  We will prove by induction that for \(i=0,1,\dotsc,\kappa\), one has:
  \[
    \int_{\widebar{X}_{\kappa}}
    f\bigl(\underline{v}\bigr)
    =
    \bigl[
      t_{\kappa-i+1}^{r}\dotsm t_{\kappa}^{r}
    \bigr]
    \int_{\widebar{X}_{\kappa-i}}
    f_i\bigl(\underline{vt}_{\,i}\bigr),
  \]
  with:
  \[
    f_i\bigl(\underline{vt}_{\,i}\bigr)
    \bydef
    f\bigl(\underline{vt}_{\,i}\bigr)\,
    \varPhi_i\bigl(\underline{vt}_{\,i}\bigr)\,
    \prod_{k=\kappa-i+1}^{\kappa}
    s_{1/t_{k}}(V_{\kappa-i}\otimes L_{\kappa-i}).
  \]
  Then, for \(i=\kappa\):
  \[
    \int_{\widebar{X}_{\kappa}}
    f\bigl(\underline{v}\bigr)
    =
    \bigl[
      t_{1}^{r}\dotsm t_{\kappa}^{r}
    \bigr]
    \int_{\widebar{X}_{0}}
    f_\kappa\bigl(\underline{t}\bigr),
  \]
  with:
  \[
    f_\kappa\bigl(\underline{t}\bigr)
    \bydef
    f\bigl(\underline{t}\bigr)\,
    \varPhi_\kappa\bigl(\underline{t}\bigr)\,
    \prod_{k=1}^{\kappa}
    s_{1/t_{k}}(V_{0}\otimes L_{0}).
  \]
  That is the desired formula because \(L_0=\mathcal{O}_{\widebar{X}_0}\).

  For \(i=0\), this is tautological. Now, assume that the formula holds for the index \(i\), that is to say:
  \[
    \tag{\(*\)}
    \int_{\widebar{X}_{\kappa}}
    f\bigl(\underline{v}\bigr)
    =
    \bigl[
      t_{\kappa-i+1}^{r}\dotsm t_{\kappa}^{r}
    \bigr]
    \int_{\widebar{X}_{\kappa-i}}
    f_i\bigl(\underline{vt}_{\,i}\bigr),
  \]

  According to lemma \eqref{lem:phi}, \(f_i\) can also be written:
  \[
    f_i\bigl(\underline{vt}_{\,i}\bigr)
    =
    f\bigl(\underline{vt}_{\,i}\bigr)\,
    \varPhi_{i+1}\bigl(\underline{vt}_{\,i}\bigr)\,
    \underbrace
    {
      \prod_{k=\kappa-i+1}^{\kappa}
      s_{1/t_{k}}(V_{\kappa-(i+1)}\otimes L_{\kappa-(i+1)}).
    }_{\in\Lambda[\underline{vt}_{i+1}]}
  \]
  Now applying lemma \eqref{eq:fibInt}:
  \[
    \int_{\widebar{X}_{\kappa-i}}
    f_i\bigl(\underline{vt}_{\,i}\bigr)
    =
    \bigl[t_{\kappa-i}^{r}\bigr]
    \int_{\widebar{X}_{\kappa-i-1}}
    f_i\bigl(\underline{vt}_{\,i+1}\bigr)
    s_{1/t_{\kappa-i}}(V_{\kappa-(i+1)}\otimes L_{\kappa-(i+1)}).
  \]
  It remains to state that:
  \begin{multline*}
    f_i\bigl(\underline{vt}_{\,i+1}\bigr)
    s_{1/t_{\kappa-i}}(V_{\kappa-(i+1)}\otimes L_{\kappa-(i+1)})
    =
    \\
    f\bigl(\underline{vt}_{\,i+1}\bigr)\,
    \varPhi_{i+1}\bigl(\underline{vt}_{\,i+1}\bigr)\,
    \prod_{k=\kappa-i+1}^{\kappa}
    s_{1/t_k}(V_{\kappa-(i+1)}\otimes L_{\kappa-(i+1)})
    \\
    s_{1/t_{\kappa-i}}(V_{\kappa-(i+1)}\otimes L_{\kappa-(i+1)}).
  \end{multline*}
  Here, we recognize the expression:
  \[
    f_{i+1}\bigl(\underline{vt}_{\,i+1}\bigr)\,
    =
    f
    \bigl(\underline{vt}_{\,i+1}\bigr)\,
    \varPhi_{i+1}\bigl(\underline{vt}_{\,i+1}\bigr)\,
    \prod_{k=\kappa-i}^{\kappa}
    s_{1/t_k}(V_{\kappa-(i+1)}\otimes L_{\kappa-(i+1)}).
  \]
  Thus, we can replace the integrand in order to get:
  \[
    \int_{\widebar{X}_{\kappa-i}}
    f_i\bigl(\underline{vt}_{\,i}\bigr)
    =
    \bigl[t_{\kappa-i}^{r}\bigr]
    \int_{\widebar{X}_{\kappa-i-1}}
    f_{i+1}\bigl(\underline{vt}_{\,i+1}\bigr).
  \]
  Using the induction hypothesis \((*)\), one finally gets the desired formula, for the index \(i+1\):
  \[
    \int_{\widebar{X}_{\kappa}}
    f\bigl(\underline{v}\bigr)
    =
    \bigl[
      t_{\kappa-i+1}^{r}\dotsm t_{\kappa}^{r}
    \bigr]
    \int_{\widebar{X}_{\kappa-i}}
    f_i\bigl(\underline{vt}_{\,i}\bigr)
    =
    \bigl[
      t_{\kappa-i}^{r}\dotsm t_{\kappa}^{r}
    \bigr]
    \int_{\widebar{X}_{\kappa-(i+1)}}
    f_{i+1}\bigl(\underline{vt}_{\,i+1}\bigr).
  \]
  This complete the proof.
\end{proof}

\subsection*{Laurent series expansion of rational functions}
Before going further, we give now more details about the Laurent series expansion mentioned above in the introduction (page \pageref{par:multivariate formal series}). 

In the univariate case, we will denote by \(K(\!(t)\!)\) the space of Laurent series. Equipped with the Cauchy product, it becomes a field. Indeed, the following \textsl{geometric series formula} is valid in the ring of formal power series:
\[
  (1-X)^{\moinsun}
  =
  \sum_{i\geq0}X^i,
\]
and it allows to define the formal inverse (for the Cauchy product) of any Laurent series of the form:
\[
  \varPsi
  =
  \sum_{i\geq N}\varPsi_i\,t^i
\]
with initial coefficient \(\varPsi_N\neq0\), as follows:
\begin{equation}
  \label{eq:formal_inverse}
  \varPsi^{\moinsun}
  =
  \frac{1}{\varPsi_{N}\,t^N}\;
  \sum_{k\geq0}\biggl(-\sum_{j\geq1}\frac{\varPsi_{N+j}}{\varPsi_{N}}\,t^j\biggr)^k,
\end{equation}
because the computation of the coefficient of any power of \(t\) in the later expression involves only a finite number of appearing \(k\)-th powers. The result is indeed a Laurent series, since its support is visibly bounded from below. 

A direct consequence is that any rational function enjoys a natural Laurent expansion.
Indeed, the support of a polynomial \(Q\in K[t]\), considered as a formal series, is finite. Thus, it is naturally a Laurent series. Then, by formula \eqref{eq:formal_inverse} above, we can construct a formal inverse of \(Q\) in the field of Laurent series. Now, any rational function of the form:
\[
  \frac{P(t)}{Q(t)}
  =
  P(t)\,Q^{\moinsun}(t),
\]
with also \(P\in R[t]\), can be expanded  as a Laurent series: it suffices to use the multiplication rule \eqref{eq:CauchyProduct} in order to compute the product (in the field of Laurent series) of  the numerator \(P\) by the formal inverse "\(Q^{\moinsun}\)" obtained after using the expansion rule \eqref{eq:formal_inverse}. 
This yields an injective morphism of fields:
\[
  \varPsi^0
  \colon
  K(t)
  \hookrightarrow
  K(\!(t)\!),
\]
that we call \textsl{Laurent expansion of rational functions at the origin}. 

\medskip

In the multivariate case, in order to \emph{unequivocally} expand a rational function of several variables \(t_1,\dotsc,t_\kappa\) under the form of a generalized Laurent series, it is necessary to assign at first a total ordering to the variables \(t_i\) (consider the example of \((t_1-t_2)^{\moinsun}\)). 
Then, working step by step in the univariate setting (taking account of the ordering of the variables),
one easily convinces oneself that the successive series expansions at zero  yield an injective morphism of fields:
\[
  \varPsi^0
  \colon
  K(t_{\kappa})(t_{\kappa-1})\dotso(t_1)
  \hookrightarrow
  K(\!(t_{\kappa})\!)(\!(t_{\kappa-1})\!)\dotso(\!(t_1)\!),
\]
that we call \textsl{Laurent expansion of rational functions at the origin, under the assumption \(t_1\ll\dotsb\ll t_{n-1}\ll t_n\ll 1\)}. The map \(\varPsi^0\) is indeed injective because its image contains only summable series, therefore its left inverse is the successive summation for \(t_{\kappa},t_{\kappa-1}\dotsc,t_{1}\). 

Here, the notation \(t_1\ll t_2\ll\dotsb\ll t_\kappa\ll1\) express the idea that for two integers \(k<k'\) the variable \(t_{k}\) is infinitely smaller than any (positive or negative) power of the variable \(t_{k'}\). 
To compute the iterated Laurent series expansion of a rational function \(Q\in K(t_{1},\dotsc,t_{\kappa})\), we first expand \(Q\) at the origin as a rational function of \(t_1\), formally considering any rational expression made of constants of \(K\), and variables \(t_2,\dotsc,t_{\kappa}\) as elements of the field of coefficients. Then, when expanding the coefficients of the resulting series, we forget \(t_1\) and we have: \(t_2\ll t_3\ll \dotso t_\kappa \ll 1\). We iterate the procedure until we get \(t_{\kappa}\ll1\), that is the one-dimensional case. 

Thus, an element of the ring:
\[
K\llangle t_1,\dotsc,t_{\kappa}\rrangle
\bydef
  K(\!(t_{\kappa})\!)(\!(t_{\kappa-1})\!)\dotso(\!(t_1)\!),
\]  
should be seen as a Laurent series in \(t_1\) whose coefficients are Laurent series in \(t_2\) whose coefficients are Laurent series in \(t_3\) and so on\ldots  Accordingly such an element is called an \textsl{iterated Laurent series}.

It is a bigger space than the space of multivariate Laurent series. A formal series \(\varPsi\) is an element of \(K\llangle t_1,\dotsc,t_\kappa\rrangle\) if and only if its support is well ordered for the lexicographic order. 
This condition is clearly weaker than to be bounded from below for the standard product order on \(\Z^{\kappa}\) (consider again the example of \((t_1-t_2)^{\moinsun}\)).

\medskip

We can extend the coefficient extraction operator to the field of rational functions \(K(t_1,\dotsc,t_\kappa)\) by using the injection \(\varPsi^0\). 
For a rational function \(Q\in K(t_1,\dotsc,t_\kappa)\), we always imply the assumption \(t_1\ll\dotsb\ll t_{\kappa}\ll1\) and we define the coefficient extraction operator:
\[
  \bigl[t_{1}^{i_{1}}\dotsm t_{\kappa}^{i_{\kappa}}\bigr]
  \bigl(Q\bigr)
  \bydef
  \bigl[t_{1}^{i_{1}}\dotsm t_{\kappa}^{i_{\kappa}}\bigr]
  \bigl(\varPsi^0(Q)\bigr).
\]
This convention in turn allows us to define the (Cauchy) product of a rational function by an iterated Laurent series, by using the same formalism as in \eqref{eq:CauchyProduct}.

\subparagraph*{Integration formula}
We are now in position to state a more tractable version of formula \eqref{thm:fibInt1}:
\begin{Thm}[Fiber Integration on the Demailly tower]
  \label{thm:fibInt2}
  For any polynomial:
  \[
    f\in H^{\bullet}(\widebar{X}_0,V_0)[t_1,\dotsc,t_{\kappa}],
  \]
  in \(\kappa\) variables \(t_1,\dotsc,t_{\kappa}\), with coefficients in the cohomology ring \(H^{\bullet}(\widebar{X}_0,V_0)\), having total degree at most \(n_\kappa\), 
  the cohomology class:
  \[
    f\bigl(\underline{v}\bigr)
    =
    f\bigl(v_1,\dotsc,v_\kappa\bigr)
    \in
    H^\smbullet\bigl(\widebar{X}_\kappa\bigr),
  \]
  can be integrated along the fibers of the projective bundle \(\widebar{X}_\kappa\to \widebar{X}_0\) according to the formula:
  \[
    \int_{\widebar{X}_{\kappa}}
    f\bigl(\underline{v}\bigr)
    =
    \bigl[
      t_1^{r}\dotsm t_{\kappa}^{r}
    \bigr]
    \biggl(
    \Phi_{\kappa}\bigl(\underline{t}\bigr)
    \int_{\widebar{X}_0}
    f\bigl(\underline{t}\bigr)\,
    s_{1/t_{1}}(V_0)
    \dotsm
    s_{1/t_{\kappa}}(V_0)
    \biggr),
  \]
  where \(\Phi_\kappa\) is the universal rational function:
  \[
    \Phi_{\kappa}
    (t_1,\dotsc,t_\kappa)
    =
    \prod_{1\leq i<j\leq\kappa}
    \frac{t_j-t_i}{t_j-2\,t_i}
    \;
    \prod_{2\leq i<j\leq\kappa}
    \frac{t_j-2\,t_i}{t_j-2\,t_i+t_{i-1}}.
  \]
\end{Thm}
\begin{proof}
  The product \(\Phi_{\kappa}\) can be reshaped as follows:
  \[
    \Phi_{\kappa}
    (t_1,\dotsc,t_\kappa)
    =
    \prod_{j=2}^{\kappa-1}
    \underbrace{
      \frac{t_j-t_1}{t_j-2\,t_1}
    }_{\Phi_{1,j}(\underline{t})}
    \;
    \prod_{i=2}^{j-1}
    \underbrace{
      \frac{t_j-t_i}{t_j-2\,t_i+t_{i-1}}
    }_{\Phi_{i,j}(\underline{t})}.
  \]
  For \(1\leq i<j\leq\kappa-1\), let:
  \[
    R_{i,j}(\underline{t})
    \bydef
    \varPsi^0(\Phi_{i,j}(\underline{t}))
    -
    \varPhi_{i,j}(\underline{t}).
  \]
  By coming back to the definitions, it is immediate to see that these remainders are:
  \[
    R_{1,j}(\underline{t})
    =
    \left(1-\frac{t_i}{t_j}\right)
    \sum_{k>n_{\kappa-1}}\frac{(2t_i)^k}{t_j^k}
  \]
  and:
  \[
    R_{i,j}(\underline{t})
    =
    \left(1-\frac{t_i}{t_j}\right)
    \sum_{k>n_{\kappa-1}}\frac{(2t_i-t_{i-1})^k}{t_j^k}
    \quad{\scriptstyle(2\leq i<j\leq\kappa-1)}.
  \]
  Noteworthy, the supports of these remainders satisfy:
  \[
    \mathrm{supp}\,R_{i,j}
    \subset
    \bigl\{
      (i_1,\dotsc,i_\kappa)
      \colon
      i_j<-n_{\kappa-1}
    \bigr\}
    \cap
    \bigcap_{k>j}
    \bigl\{
      (i_1,\dotsc,i_\kappa)
      \colon
      i_k=0
    \bigr\}.
  \]
  One can write:
  \[
      \varPsi^0(\Phi(\underline{t}))
      =
      \prod_{\text{entries}}
      \xym{4,4}{
            1\ar@{.}"m4,4"!UL+<20pt,-5pt>\ar@{.}[3,0]&\varPhi_{1,2}+R_{1,2}\ar@{.}[2,2]\ar@{.}[0,2]&&\varPhi_{1,\kappa}+R_{1,\kappa}\ar@{.}[2,0]\\
            &&&&\\
            &&&\varPhi_{\scriptscriptstyle\kappa-1,\kappa}+R_{\scriptscriptstyle\kappa-1,\kappa}\\
            \hskip10pt 1\hskip10pt\ar@{.}[]!R-<10pt,0pt>;"m4,4"!L+<20pt,0pt>&&&\hskip22pt 1\hskip22pt
          }.
  \]
  We clean inductively this array of the remainder in the \(k\)-th column.
  Let \(\mathsf{array}_j\) be the array deduced from the above array by dropping the remainders in the \(j\) last columns.
  We will show that:
  \[
    \bigl[
      t_1^{r}\dotsm t_{\kappa}^{r}
    \bigr]
    \Bigl(
    \prod\mathsf{array}_j\,
    I(t_1,\dotsc,t_\kappa)
    \Bigr)
    =
    \bigl[
      t_1^{r}\dotsm t_{\kappa}^{r}
    \bigr]
    \Bigl(
    \prod\mathsf{array}_{j+1}\,
    I(t_1,\dotsc,t_\kappa)
    \Bigr)
  \]
  where:
  \[
    I(t_1,\dotsc,t_\kappa)
    \bydef
    \int_{\widebar{X}_0}
    f(t_1,\dotsc,t_\kappa)\,
    s_{1/t_1}(V_0)\dotsm s_{1/t_{\kappa}}(V_0).
  \]

  We generalize the notation \(I(t_{1},\dotsc,t_{\kappa})\) by setting:
  \[
    I(t_1,\dotsc,t_i)
    \bydef
    \bigl[t_{i+1}^r\dotsm t_{\kappa}^r]
    \bigg(
    I(\underline{t})\,
    \prod_{k=i+1}^{\kappa-1}
    \prod_{j=1}^{k-1}
    \varPhi_{j,k}(\underline{t})
    \bigg)
    \qquad{\scriptstyle(i=0,1,\dotsc,\kappa)}.
  \]
  Then we claim (find the proof below):
  \[
    \mathrm{supp}\,I(t_1,\dotsc,t_{\kappa-j})
    \subset
    \bigl\{
      (i_1,\dotsc,i_{\kappa-j})
      \colon
      i_{\kappa-j}\leq n_{\kappa}\,
    \bigr\}.
  \]

  On the other hand, one is easily convinced that:
  \[
    \prod\mathsf{array}_j
    =
    \prod\mathsf{array}_{j+1}
    +
    R_j\;
    \prod_{l=\kappa-j+1}^{\kappa-1}
    \prod_{k=1}^{l-1}
    \varPhi_{k,l}(\underline{t}),
  \]
  where \(R_{j}\) is a given series such that:
  \[
    \mathrm{supp}\,R_j(t_1,\dotsc,t_{\kappa-j})
    \subset
    \bigl\{
      (i_1,\dotsc,i_{\kappa-j})
      \colon
      i_{\kappa-j}<-n_{\kappa-1}
    \bigr\}.
  \]

  It is now clear that the second part cannot contribute to the coefficient of \(t_{\kappa-j}^r\) in:
  \[
    [t_{\kappa-j+1}^r\dotsm t_{\kappa}^r]
    \Bigl(\prod\mathsf{array}_j\,I(\underline{t})\Bigr),
  \]
  because the degree of \(t_{\kappa-j}\) in the corresponding term: 
  \[
    [t_{\kappa-j+1}^r\dotsm t_{\kappa}^r]
    \bigg(
    R_j(t_1,\dotsc,t_{\kappa-j})\;
    \prod_{l=\kappa-j+1}^{\kappa-1}
    \prod_{k=1}^{l-1}
    \varPhi_{k,l}(\underline{t})\,
    I(\underline{t})
    \bigg)
    =
    R_j(t_1,\dotsc,t_{\kappa-j})\;
    I(t_1,\dotsc,t_{\kappa-j})
  \]
  is \emph{strictly} less than
  \(
    -n_{\kappa-1}
    +
    n_\kappa
    =
    r
  \).

  Consequently:
  \[
    \bigl[
      t_{\kappa-j}^{r}\dotsm t_{\kappa}^{r}
    \bigr]
    \Bigl(
    \prod\mathsf{array}_j\,
    I(t_1,\dotsc,t_\kappa)
    \Bigr)
    =
    \bigl[
      t_{\kappa-j}^{r}\dotsm t_{\kappa}^{r}
    \bigr]
    \Bigl(
    \prod\mathsf{array}_{j+1}\,
    I(t_1,\dotsc,t_\kappa)
    \Bigr),
  \]
  and by extraction of the coefficient of the monomial \(t_1^r\dotsm t_{\kappa-j-1}^r\), as announced:
  \[
    \bigl[
      t_{1}^{r}\dotsm t_{\kappa}^{r}
    \bigr]
    \Bigl(
    \prod\mathsf{array}_j\,
    I(t_1,\dotsc,t_\kappa)
    \Bigr)
    =
    \bigl[
      t_{1}^{r}\dotsm t_{\kappa}^{r}
    \bigr]
    \Bigl(
    \prod\mathsf{array}_{j+1}\,
    I(t_1,\dotsc,t_\kappa)
    \Bigr).
  \]

  An induction finishes the proof because:
  \[
    \varPsi^0(\Phi_{\kappa})
    =
    \prod\mathsf{array}_{0}
    \quad
    \text{and}
    \quad
    \varPhi_{\kappa}
    =
    \prod\mathsf{array}_{\kappa}.
    \qedhere
  \]
\end{proof}
We have added a lot of non contributive terms, however in practice (\cite{arXiv:1402.1396}), the above reformulation of \eqref{thm:fibInt1} is more efficient, because it takes account of the convergence of the series at stake.

Finally, we prove our claim above in the proof, that:
  \[
    \mathrm{supp}\,I(t_1,\dotsc,t_{\kappa-j})
    \subset
    \bigl\{
      (i_1,\dotsc,i_{\kappa-j})
      \colon
      i_{\kappa-j}\leq n_{\kappa}\,
    \bigr\}.
  \]
  Actually we will be more precise and show that for \(j=1,\dotsc,\kappa\):
  \[
    \mathrm{supp}\,I(t_1,\dotsc,t_j)
    \subset
    \bigl\{
      (i_1,\dotsc,i_j)
      \colon
      i_j\leq n_{j}\,
    \bigr\}.
  \]
\begin{proof}
  In order to prove this statement, it is easier to work with genuine polynomials, and not Laurent polynomials.
  An important remark is that for any two integers \(k<l\), the Laurent polynomial \(t_{l}^{n_{\kappa-1}}\varPhi_{k,l}(\underline{t})\) is in fact a genuine polynomial, having degree \(n_{\kappa-1}\).
  Thus, we rather consider:
  \[
    t_j^n
    I(t_1,\dotsc,t_j)
    =
    \bigl[t_{j+1}^{m_{j+1}}\dotsm t_{\kappa-1}^{m_{\kappa-1}}\,t_{\kappa}^{n+r}]
    \bigg(
    \bigl(
    t_j^n\dotsm t_{\kappa}^n\,
    I(\underline{t})\,
    \bigr)
    \prod_{l=j+1}^{\kappa-1}
    \prod_{k=1}^{l-1}
    \bigl(
    t_{l}^{n_{\kappa-1}}
    \varPhi_{k,l}(\underline{t})
    \bigr)
    \bigg),
  \]
  where
  \(
    m_l
    =
    n+r+(l-1)(n_{\kappa-1})
  \).

  The appearing polynomial in \(t_j,t_{j+1},\dotsc,t_\kappa\) has degree:
  \[
    \deg_{t_j,\dotsc,t_{\kappa}}
    \bigg(
    \bigl(
    t_j^n\dotsm t_{\kappa}^n\,
    I(\underline{t})\,
    \bigr)
    \prod_{l=j+1}^{\kappa-1}
    \prod_{k=1}^{l-1}
    \bigl(
    t_{l}^{n_{\kappa-1}}
    \varPhi_{k,l}(\underline{t})
    \bigr)
    \bigg)
    \leq
    n_{\kappa}
    +
    n
    +
    \sum_{l=j+1}^{\kappa-1}
    (m_l-r)
    +
    n.
  \]
  Extracting the coefficient of \(\bigl[t_{j+1}^{m_{j+1}}\dotsm t_{\kappa-1}^{m_{\kappa-1}}\,t_{\kappa}^{n+r}]\) decrease the degree by at least:
  \[
    \sum_{l=j+1}^{\kappa-1}
    m_l
    +(n+r).
  \]
  Finally we get a polynomial in \(t_j\) having degree:
  \[
    \deg_{t_j}
    \left(
    t_j^n
    I(t_1,\dotsc,t_j)
    \right)
    \leq
    n
    +
    n_{\kappa}
    -
    r
    -
    \sum_{l=j+1}^{\kappa-1}
    r
    =
    n+n_j.
  \]
  Thus, as announced:
  \[
    \deg_{t_j}
    \bigl(
    I(t_1,\dotsc,t_j)
    \bigr)
    \leq
    n_{j}.
    \qedhere
  \]
\end{proof}

Notice that our theorem holds as well without the (natural) technical assumption on the degree of \(f\).

%
%BIBLIOGRAPHY
%
\bibliographystyle{amsplain}
\bibliography{$LATEX/these.bib}
\vfill
\end{document}